\documentclass[12pt]{amsart}

\usepackage[hmargin=2.5cm,vmargin=2.5cm]{geometry}
\usepackage{graphicx}
\usepackage{color}
\usepackage{braket}
\usepackage{hyperref}

\DeclareMathOperator{\dom}{dom}

\DeclareMathOperator{\Ran}{Ran}
\DeclareMathOperator{\Proj}{Proj}
\DeclareMathOperator{\Ker}{Ker}

\DeclareMathOperator{\Real}{Re}

\newtheorem{theorem}{Theorem}[section]
\newtheorem{lemma}[theorem]{Lemma}

\newtheorem{proposition}[theorem]{Proposition}

\theoremstyle{remark}

\newtheorem{example}[theorem]{Example}

\numberwithin{equation}{section}

\newcommand{\C}{\mathbb{C}}
\newcommand{\Edges}{\mathcal{E}}

\newcommand{\trD}{{\gamma_D}}

\newcommand\cH{{\mathcal{H}}}
\newcommand\cS{{\mathcal{S}}}
\newcommand\cF{{\mathcal{F}}}
\newcommand\cL{{\mathcal{L}}}

\begin{document}

\title{Exotic eigenvalues of shrinking metric graphs}
\date{\today}
\author{Gregory Berkolaiko}
\address{Department of Mathematics, Texas A\&M University, College
  Station, TX 77843-3368, USA}
\author{Yves Colin de Verdi\`ere}
\address{Institut Fourier, 
Universit\'e de Grenoble I,
38402 Saint-Martin d'H\`eres, France}

\maketitle

\begin{abstract}
  Eigenvalue spectrum of the Laplacian on a metric graph with
  arbitrary but fixed vertex conditions is investigated in the limit
  as the lengths of all edges decrease to zero at the same rate.  It
  is proved that there are exactly four possible types of eigenvalue
  asymptotics.  The number of eigenvalues of each type is expressed
  via the index and nullity of a form defined in terms of the vertex
  conditions.
\end{abstract}

\section{Introduction}
\label{sec:intro}

Analysis of differential operators on a metric graph or network is a
burgeoning area of research due to its numerous applications in
physics and engineering.  The first step in solving, say, a control
problem on a metric network \cite{CasZua_siamjco} is to understand the
spectral properties of the time-independent operator, which in many
cases is the Laplacian or Sturm--Liouville operator on every edge
augmented by suitable matching conditions on every vertex
\cite{BerKuc_graphs,Mugnolo_book}.

Many interesting mathematical problems arise here due to richness of
the topology of the network and of the physically realizable
\cite{CheExnTur_anp10} matching conditions.  To give an example, the
Laplacian on a metric graph can have eigenmodes localizing exactly on
a sub-structure, such as a cycle or a path connecting two vertices of
degree one \cite{SchKot_prl03,CdV_ahp15}.  Such eigenmodes play havoc
with inverse and/or control problems: a localized eigenmode is not
detectable or controllable via the edges where it is identically zero
\cite{Kur_jmp13}.  Fortunately, for the Laplacian with the so-called
standard (or Neumann--Kirchhoff) matching conditions, such exactly
localized eigenstate disappear under small perturbations of the edge
metric \cite{BerLiu_jmaa17} --- with some important exceptions.

Until recently, mathematical analyses of the Laplacian on a metric
graph usually assumed a uniform lower bound on the edge length.
However, spectral properties of graph models with short lengths are
important in applications, such as in creation of periodic materials
with special wave transmission properties
\cite{Kuc_jpa05,BarExnTat_rmp21,Che+_prep22,LawTanChr_sr22}.  Limits
of operators on metric graphs with short edges and general matching conditions
started\footnote{Spectral analysis of the \emph{discrete Laplacian} on a
  graph with shrinking edges has a more substantial history, including
works such as \cite{CdV_jctb90,CdV_prep18}.} receiving significant analytical attention with concurrent
publication of the works by Cacciapuoti \cite{Cac_s19} and by
Berkolaiko, Latushkin and Sukhtaiev \cite{BerLatSuk_am19}.  These
publications treated slightly different models, which were
subsequently generalized by Borisov \cite{Bor_am22}.  In these works,
convergence of operators (in the norm-resolvent sense) is established
under certain non-resonance conditions.  Counter-examples to
convergence use a clever choice of matching conditions to produce what
we call ``exotic states'': eigenfunction that localize on the
shrinking edges of the graph.  The limiting operator has those edges
shrunk to a point which cannot support eigenfunctions.  The
non-resonance conditions (such as \cite[Cond.~3.2]{BerLatSuk_am19}
reviewed in equation~\eqref{eq:NRC} below) are sufficient to exclude
exotic eigenfunctions, but to what extent such conditions are
necessary is an open problem.

In this note, we consider perhaps the simplest model of a graph with
short edge lengths --- the Laplacian on a graph where \emph{all}
lengths are shrinking uniformly.  The aim of this simplification is to
gain complete understanding of the number of the exotic states, and of
the eigenvalue asymptotics in general.  In particular, we find that,
in this context, the non-resonance condition
\cite[Cond.~3.2]{BerLatSuk_am19} is both necessary and sufficient in
that it provides the exact count of the exotic states.

To be more precise, we consider the Laplacian on a metric graph
$\Gamma_\epsilon$ with arbitrary but fixed self-adjoint matching
conditions and edge lengths that are uniformly shrinking: the length
of an edge $e$ is $\epsilon \ell_e$, where $\ell_e>0$ is fixed.
Intuitively, a shorter string produces a higher pitch, so the
eigenvalues should increase as $\epsilon \to 0$.  From dimensional
considerations of the eigenvalue equation
$-\frac{d^2}{dx^2}f = \lambda f$, one can expect the asymptotic
behavior $\lambda \sim \epsilon^{-2}$.  This intuition very quickly
runs into trouble: a Neumann Laplacian on an interval has a constant
eigenvalue $0$, Robin Laplacian may have eigenvalues scaling as
$\epsilon^{-1}$ (positive or negative), or even converging to a
non-zero constant, see examples is Appendix~\ref{sec:examples}.  In
this note we explain why these are the only options in our setting and
give a quantitative description of different behaviors.

\subsection*{Acknowledgements}

We are grateful to Jens Bolte, Pavel Exner, Yuri Latushkin, Uzy
Smilansky, Selim Sukhtaiev who at various times guided us with their
intuition on this problem.  Part of the work was done during GB's
visit to Institut Fourier in Grenoble, France.  Institut's hospitality
is gratefully acknowledged.  GB was partially supported by the
National Science Foundation grant DMS-1815075.

\section{Notation and the main result}

Denoting by $\Edges$ the finite edge set of the graph in question, we
consider the operator $H_\epsilon$ acting as $-d^2/dx^2$ in the space
$\cH_\epsilon := \bigoplus_{e\in \Edges} L_2(0,\epsilon \ell_e)$,
where $\ell_e>0$ are assumed to be fixed.  The domain of the operator
$H_\epsilon$ is an appropriate subspace of $\cH_\epsilon$ which we now
describe.

Let $E \simeq \C^{2|\Edges|}$ be the space of boundary values of
functions on graph edges and $E' \simeq \C^{2|\Edges|}$ be the space
of derivative values.  More precisely, denoting by $e^-$ and $e^+$ the
two endpoints of the edge $e\in\Edges$, we set
\begin{equation}
  \label{eq:traces_def}
  f_{e^-} = f(0), \qquad f_{e^+} = f(\epsilon\ell_e),
  \qquad \partial f_{e^-} = f'(0),
  \qquad \partial f_{e^+} = -f'(\epsilon\ell_e),
\end{equation}
which are well-defined when $f \in H^2(0,\epsilon\ell_e)$.
Corresponding to a function $f$, we then have two vectors
\begin{equation}
  \label{eq:trace_vectors}
  F := \big(f_{e^\pm}\big)_{e\in\Edges} \in E
  \qquad\text{and}\qquad
  F := \big(\partial f_{e^\pm}\big)_{e\in\Edges} \in E'.  
\end{equation}
The edge endpoints are ordered arbitrarily but in the same way for $F$
and $F'$.  The mapping from $f \in \bigoplus_{e\in \Edges}
H^2(0,\epsilon\ell_e)$ to $F \in E$ will be denoted by $\trD$.

Self-adjoint realizations of $-d^2/dx^2$ may be described using the
theory of boundary triplets
\cite{Schmudgen_unboundedSAO,BehHasDeS_boundarytriples} which use
Green's formula to define a symplectic structure on $E\oplus E'$,
with $E'$ being viewed as a dual space to $E$ --- with a natural
identification between them.  With respect to this (complex!)
symplectic structure, the self-adjoint realizations of $-d^2/dx^2$ are
in one-to-one correspondence with the Lagrangian subspaces of $E\oplus
E'$ \cite{Shm_ian74,KosSch_jpa99,Har_jpa00}.  However, in our setting, more
directly usable parametrization of self-adjoint matching conditions
can be written as
\begin{align}
  \label{eq:sa_cond_D}
  &P_D F = 0,\\
  \label{eq:sa_cond_N}
  &P_N F' = 0,\\
  \label{eq:sa_cond_R}
  &P_R F' = \Lambda P_R F.
\end{align}
Here $P_D$, $P_N$ and $P_R$ are orthogonal projectors such as
$P_D+P_N+P_R=I$ (in particular, the projectors are mutually
orthogonal); $\Lambda$ is an invertible self-adjoint operator acting
on the space $P_R E$; equation~\eqref{eq:sa_cond_R} assumes the
natural identification of $E'$ and $E$.  Some examples of the matching
conditions can be found in Appendix~\ref{sec:examples}.  We remark that the
connectivity among edges (i.e. ``vertices'') is encoded in the
matching conditions \eqref{eq:sa_cond_D}-\eqref{eq:sa_cond_R} and will
not play a direct role in the subsequent analysis.  For this reason we
did not introduce the notion of a vertex and use the term ``matching
conditions'' instead of perhaps more widespread ``vertex
conditions''.

Let
\begin{align}
  \label{eq:D0def}
  D_0
  &:= \{F \in E \colon F_{e^-} = F_{e^+}\ \text{ for all }e\in \Edges\},
  \\
  \label{eq:N0def}
  N_0
  &:= \{F' \in E' \colon F_{e^-}' = -F_{e^+}'\ \text{ for all }e\in \Edges\}.
\end{align}
Our main results is the following.

\begin{theorem}
  \label{thm:main}
  Let
  \begin{equation}
    \label{eq:F0_alt_def}
    \cF_0 := D_0 \cap \Ker P_D,
    \qquad m_0 := \dim \cF_0.
  \end{equation}
  Denote by $n_-$, $n_0$ and $n_+$ the number of negative, zero and
  positive eigenvalues of the sesquilinear form
  $q[G,F] := \langle G, P_R \Lambda P_R F\rangle$ restricted to
  $\cF_0$.
  Then the spectrum $\{\lambda_n(\epsilon)\}_{n\in \mathbb{N}}$ of
  $H_\epsilon$ consists of
  \begin{enumerate}
  \item \label{item:fast}
    Infinitely many ``normal'' or ``fast'' eigenvalues which tend
    to $+\infty$ at the rate $1/\epsilon^2$,
    \begin{equation*}
      \lambda_n(\epsilon) = \mu_n \epsilon^{-2} + O(\epsilon^{-1}),
      \quad
      n > m_0,
      \qquad
      \epsilon \to 0,
    \end{equation*}
  \item \label{item:slow}
    and $m_0$ lowest ``slow'' eigenvalues, which
    are further boken into three types,
    \begin{enumerate}
    \item \label{item:slow_neg}
      $n_-$ ``slow negative'' eigenvalues tending to $-\infty$ at
      the rate $1/\epsilon$,
    \item \label{item:slow_fin}
      $n_0$ ``exotic'' eigenvalues tending to a finite
      \emph{non-positive} limit,
    \item \label{item:slow_pos}
      $n_+$ ``slow negative'' eigenvalues tending to $+\infty$
      at the rate $1/\epsilon$.
    \end{enumerate}
  \end{enumerate}
\end{theorem}

It is interesting to compare this result with the convergence results
established in \cite[Thm.~3.6]{BerLatSuk_am19}.  That theorem predicts
spectral convergence on any fixed interval, to the spectrum of the
Laplacian on a metric graph which, in the present context, is empty.
The sufficient condition for this convergence is the ``Non-Resonance
Condition'' \cite[Cond.~3.2]{BerLatSuk_am19}, which, adjusted to the
present context, reads
\begin{equation}
  \label{eq:NRC}
  \dim \Proj_{E}\big(\cL \cap (D_0 \oplus N_0)\big) = 0.
\end{equation}
Here $\cL$ is the subspace of $E \oplus E'$ consisting of $F$ and $F'$
satisfying matching conditions \eqref{eq:sa_cond_D}-\eqref{eq:sa_cond_R}
(assumed to be independent of $\epsilon$).  Comparing with
part~\eqref{item:slow}\eqref{item:slow_fin} of Theorem~\ref{thm:main},
we conclude that the Non-Resonance Condition~\eqref{eq:NRC} should be
the sufficient condition for the absense of exotic eigenvalues.  It
turns out to be a necessary condition as well, as shown by the
following.

\begin{proposition}
  \label{prop:NRCns}
  The space $\Proj_{E}\big(\cL \cap (D_0 \oplus N_0)\big)$ is the null
  space of the sesquilinear form
  $q[G,F] := \langle G, P_R \Lambda P_R F\rangle$ restricted to
  $\cF_0$.  In particular, a uniformly shrinking graph
  $H_\epsilon$ has no exotic eigenvalues if and only if the
  Non-Resonance Condition~\eqref{eq:NRC} holds.
\end{proposition}

\section{Counting eigenvalues}

Rescaling the edges to constant lengths we obtain an operator $R_\epsilon$
on a fixed graph $\Gamma_1$ with lengths
$\big(\ell_e\big)_{e\in\Edges}$.  It acts as $-d^2/dx^2$ but with
matching conditions that now depend on $\epsilon$, 
\begin{align}
  \label{eq:sa_cond_D_rescaled}
  &P_D F = 0,\\
  \label{eq:sa_cond_N_rescaled}
  &P_N F' = 0,\\
  \label{eq:sa_cond_R_rescaled}
  &P_R F' = \epsilon \Lambda P_R F.
\end{align}
The eigenvalues of $H_\epsilon$ and $R_\epsilon$ are related by
\begin{equation}
  \label{eq:eig_rescale}
  \lambda\big(H_\epsilon\big) = \frac{\lambda(R_\epsilon)}{\epsilon^2}.
\end{equation}

We can now view the rescaled operator $R_\epsilon$ as a regular
perturbation problem for the operator $R_0$, obtained by simply
setting $\epsilon=0$ in
\eqref{eq:sa_cond_D_rescaled}-\eqref{eq:sa_cond_R_rescaled}.  The
fast eigenvalues of Theorem~\ref{thm:main} part (\ref{item:fast})
correspond to non-zero eigenvalues of $R_0$.  The various slow
eigenvalues branch out of the zero eigenspace of $R_0$.

We will show that $m_0 = \dim\Ker R_0$ and use Rellich--Kato Selection
Theorem \cite[Thms.\ II.5.4 and VII.3.9]{Kato_perturbation} (known as
Hellmann-Feynman formula for degenerate eigenvalues in the physics
literature) to classify how the eigenvalues branch out from zero.
This information is contained in the spectrum of the restriction of
$q$ to $\Ker R_0$.  More precisely, negative and positive eigenvalues
of the restricted $q$ correspond to ``slow negative'' and ``slow
positive'' eigenvalues of Theorem~\ref{thm:main} part
(\ref{item:fast}).  Zero eigenvalues of the restricted $q$ are
responsible for the ``exotic'' eigenvalues in Theorem~\ref{thm:main}
part (\ref{item:slow})

\subsection{Spectrum of the rescaled unperturbed problem.}

We first aim to understand the unperturbed problem
$R_0$.  Setting $\epsilon=0$ in the matching conditions
\eqref{eq:sa_cond_D_rescaled}-\eqref{eq:sa_cond_R_rescaled}, we
obtain the matching conditions for the unperturbed problem:
\begin{align}
  \label{eq:sa_cond_D_rescaled_unperturbed}
  &P_D F = 0,\\
  \label{eq:sa_cond_N_rescaled_unperturbed}
  &(I-P_D) F' = 0,
\end{align}
where we used that $P_D + P_N + P_R = I$.

\begin{lemma}
  \label{lem:unperturbed_kernel}
  The rescaled unperturbed operator $R_0$ is a non-negative
  self-adjoint operator with discrete spectrum.  The kernel of $R_0$
  is the span of all edgewise constant functions satisfying
  $P_D F = 0$.  Its dimension is equal to
  $m_0 = \dim\left(D_0 \cap \Ker P_D\right)$.
\end{lemma}

\begin{proof}
  For the purpose of establishing Lemma~\ref{lem:unperturbed_kernel},
  it is much simpler to work with the quadratic form
  corresponding to the operator $R_0$.  The form $Q_0$ is
  \begin{equation}
    \label{eq:Q0}
    Q_0[f] = \int_{\Gamma_1} |f(x)'|^2 dx,
    \qquad
    \dom[Q_0] = \left\{f \in \bigoplus_{e\in \Edges} H^1(0,\ell_e)
    \colon P_D F = 0 \right\}.
  \end{equation}
  It is immediate that $Q_0 \geq 0$.  That this form corresponds to a
  self-adjoint operator on the graph $\Gamma_1$ is a well-known fact
  (see, e.g., \cite[Sec.~1.4.3]{BerKuc_graphs}).  The spectrum is
  discrete due to compactness of $\Gamma_1$
  \cite[Sec.~3.1.1]{BerKuc_graphs}.

  To be in the kernel of $Q_0$, the function $f$ must have
  $f'\equiv0$, as can be seen directly from \eqref{eq:Q0}.  We
  conclude that $f$ is edge-wise constant, which immediately gives
  $F \in D_0$; $f$ also must belong to the domain of $Q_0$, i.e.\
  satisfy the condition $P_DF=0$.

  Conversely, any $G \in  D_0 \cap \Ker P_D$ corresponds to an
  edge-wise constant function $g$ which belongs to the domain of $R_0$
  and obviously satisfies $R_0g = 0$.
\end{proof}

\subsection{Proof of Theorem~\ref{thm:main}}
\label{sec:main_proof}

From \cite{BerKuc_incol12} (see also \cite[Sec.~3.1.2]{BerKuc_graphs})
we know that the eigenvalues of $R_\epsilon$ depend analytically on
the parameter $\epsilon$.  In particular, for all $n > m_0$, the
$n$-th eigenvalue can be represented as
\begin{equation}
  \label{eq:Re_eig_expansion0}
  \lambda_n\big(R_\epsilon\big) = \lambda^0_n + O(\epsilon),
  \qquad
  \lambda^0_n := \lambda_n\big(R_0\big) > 0.
\end{equation}
Note that here we have perturbation with one parameter only.  In case
of eigenvalue multiplicity of $R_0$ we are effectively using the
Rellich--Kato Selection Theorem, but only from one
side, $\epsilon>0$.  Part~(\ref{item:fast}) of
Theorem~\ref{thm:main} now follows from \eqref{eq:eig_rescale}.

The lowest $m_0$ eigenvalues bifurcate from the eigenvalue 0 of $R_0$.
Here we apply the Rellich--Kato Selection Theorem in earnest, to
determine the slopes of eigenvalue bifurcation.  The appropriate version
of the theorem, for perturbations in the domain of the self-adjoint
operator, was established only recently
\cite[Thm.~3.23]{LatSuk_prep20} (see also
\cite[Thm.~2.10]{LatSuk_blms226}).  In our case, it boils down to
finding the eigenvalues of the derivative of the
quadratic form
\begin{equation}
  \label{eq:Qepsilon}
  Q_\epsilon[f] = \int_{\Gamma_1} |f(x)'|^2 dx
  + \epsilon \big\langle P_R (\trD f), \Lambda P_R (\trD f)\big\rangle 
\end{equation}
of the operator $R_\epsilon$ on the kernel of the operator $R_0$.  By
Lemma~\ref{lem:unperturbed_kernel}, the
eigenvalues of the quadratic form
$\big\langle P_R (\trD f), \Lambda P_R (\trD f)\big\rangle$ on
$\Ker R_0$ are the same as the eigenvalues of
$\big\langle P_R F, \Lambda P_R F\big\rangle$ on $D_0 \cap \Ker P_D$.
These eigenvalues, which we denote by $\{\mu_n\}_{n=1}^{m_0}$, yield
the next term in the expansion of $\lambda_n(R_\epsilon)$ for
$n \leq m_0$,
\begin{equation}
  \label{eq:Re_eig_expansion1}
  \lambda_n\big(R_\epsilon\big)
  = 0 + \mu_n\epsilon + O(\epsilon^2),
  \qquad
  n \leq m_0.
\end{equation}
Combined with \eqref{eq:eig_rescale}, we get part~(\ref{item:slow}) of
Theorem~\ref{thm:main}.

Finally, for $n$ such that $\mu_n=0$, we want to understand the sign
of the $\epsilon^2$ term in the expansion
\eqref{eq:Re_eig_expansion1}.  We know a priori
\cite{BerKuc_incol12,KucZha_jmp19} that the eigenvalues and
eigenfunctions are analytic in $\epsilon$ (or can be chosen to be so
in the cases of multiplicity).  Let $n \in \{n_-+1, \ldots, n_-+n_0\}$
and consider an eigenvalue with small $\epsilon$ expansion
\begin{equation}
  \label{eq:Re_eig_expansion2}
  \lambda_n\big(R_\epsilon\big)
  = 0 + \nu_n\epsilon^2 + O(\epsilon^3),
\end{equation}
and let
\begin{equation}
  \label{eq:efun_expansion}
  f(x;\epsilon) = f_0(x) + \epsilon f_1(x) + O(\epsilon^2)
\end{equation}
be the corresponding eigenfunction.  We will collect some information
about $f_0$ and $f_1$ and then obtain $\lambda_n\big(R_\epsilon\big)$
as $Q_\epsilon[f(x;\epsilon)]$.

Since $f_0 \in \Ker R_0$, we know it is edgewise constant.  Since
$-\frac{d^2}{dx^2} f(x; \epsilon) = \lambda_n f(x;\epsilon)$, we get
that $-\frac{d^2}{dx^2} f_1 = 0$ and therefore
\begin{equation}
  \label{eq:f_1_edges}
  f_1\big|_e = b_ex + c_e.
\end{equation}
Let $F_1' \in E'$, $F_1\in E$ and $F_0 \in E$ be the vectors of the
Neumann values of $f_1$ and the Dirichlet values of $f_1$ and $f_0$,
correspondingly (as remarked above, Neumann values of $f_0$ are zero).
Since the eigenfunction $f(x;\epsilon)$ satisfies the
$\epsilon$-dependent matching conditions \eqref{eq:sa_cond_D_rescaled}-\eqref{eq:sa_cond_R_rescaled},
we have
\begin{equation}
  \label{eq:PR10}
  P_D F_1 = 0, \quad
  P_N F_1' = 0, \quad
  P_R F_1' = \Lambda P_R F_0,
\end{equation}
which will become useful shortly.

Substituting expansions \eqref{eq:Re_eig_expansion2} and
\eqref{eq:efun_expansion} into $\lambda_n\big(R_\epsilon\big) =
Q_\epsilon[f(x;\epsilon)]$ and extracting the term of order
$\epsilon^2$, we get
\begin{align}
  \label{eq:e2terms}
  \nu_n
  &= \sum_{e\in\Edges} \int_{\ell_e} |b_e|^2 dx
    + \left<P_R F_0, \Lambda P_R F_1 \right>
    + \left<P_R F_1, \Lambda P_R F_0 \right>
  \\
  \label{eq:nu_expansion}
  &=
    \sum_{e\in\Edges} \ell_e |b_e|^2
    + 2 \Real \left<P_R F_1, P_R F_1' \right>,
\end{align}
where we used \eqref{eq:PR10} to get rid of $\Lambda P_R F_0$ (note that
$\Lambda$ is Hermitian).  We can simplify the inner product in
\eqref{eq:nu_expansion} further by noting that
\begin{align*}
  P_R F_1 &= (I - P_D - P_N) F_1 = (I-P_N) F_1,
  \\
  P_R F_1' &= (I-P_D - P_N) F_1' = (I-P_D) F_1',
\end{align*}
and therefore
\begin{equation}
  \label{eq:PR_removed2}
  \left<P_R F_1, P_R F_1' \right> = 
  \left< (I-P_N) F_1, (I-P_D) F_1' \right> =
  \left< F_1, F_1' \right>,
\end{equation}
both using \eqref{eq:PR10} as well as the mutual orthogonality of the
projectors $P_N$ and $P_D$.

We now compute $F_1$ and $F_1'$ more explicitly,
\begin{align}
  \label{eq:F1prime}
  F_1'\big|_{e^-}
  &=
    b_e,
  &
    F_1'\big|_{e^+}
  &=
    -b_e \\
  F_1\big|_{e^-}
  &=
    -\frac12 \ell_e b_e + c_e,
  &
  F_1\big|_{e^+}
  &=
    \frac12 \ell_e b_e + c_e,
\end{align}
where we assumed without loss of generality that the local coordinate
$x$ on the edge $e$ runs from $-\ell_e/2$ at the endpoint $e^{-}$ to
$\ell_e/2$ at the endpoint $e^{+}$.  In particular, we have $F_1' \in
N_0$ and
\begin{equation}
  \label{eq:F1_vector}
  F_1 = -\frac12 L F_1' + C,
\end{equation}
where $C\in D_0$ is a vector of $c_e$ and $L$ is the diagonal
$2|\Edges|\times2|\Edges|$ matrix with $\ell_e$ on the diagonal.  We
substitute \eqref{eq:F1_vector} into \eqref{eq:PR_removed2} to obtain
\begin{align}
  \label{eq:nu1}
  2 \Real \left<F_1, F_1' \right>
  &=
    2 \Real \left<-\frac12 L F_1' + C, F_1' \right>
    = - \Real \left<L F_1', F_1' \right>
  \\
  &=
    - \sum_{e\in\Edges} 2|b_e|^2 \ell_e, 
\end{align}
where we used that $C \perp F_1'$ (because $D_0 \perp N_0$) and the 2
appeared because each edge contributes two endpoints to the inner
product.

To summarize, with \eqref{eq:nu_expansion} we get
\begin{equation}
  \label{eq:nu_answer}
  \nu_n = - \sum_{e\in\Edges} 2|b_e|^2 \ell_e \geq 0,
\end{equation}
which is establishes the last part of Theorem~\ref{thm:main}, namely
that the exotic eigenvalues are non-positive.

\subsection{Counting exotic eigenvalues: proof of
  Proposition~\ref{prop:NRCns}} 
\label{sec:compare_to_nonresonance}

Assuming
\begin{equation}
  \label{eq:asssumption_ker}
  (F,F') \in \cL \cap (D_0 \oplus N_0),  
\end{equation}
we want to show then $F$ belongs to the kernel of $q[G,F]$ restricted
to $\cF_0 = D_0 \cap \Ker P_D$.  Namely, we need to check that
$F \in \cF_0$ and that
\begin{equation}
  \label{eq:annihilate_form}
  \langle P_R \Phi, \Lambda P_R F\rangle = 0
  \qquad
  \text{for all }\Phi \in \cF_0.
\end{equation}
That $F\in\cF_0$ follows directly from \eqref{eq:asssumption_ker}.
Furthermore, since $F' \in N_0 = D_0^\perp$ (assuming the natural
identification of $E'$ and $E$), we get
\begin{equation*}
  0 = \langle \Phi, F' \rangle
  = \langle P_N \Phi + P_R \Phi, P_D F' + P_R F' \rangle
  = \langle P_R \Phi, P_R F' \rangle
  = \langle P_R \Phi, \Lambda P_R F \rangle.
\end{equation*}
Here, in the second equality we used that $P_D+P_N+P_R = I$,
$P_D\Phi = 0$ and $P_N F' = 0$; in the second we used mutual
orthogonality of $P_D$, $P_N$ and $P_R$; in the third we used equation
\eqref{eq:sa_cond_R}.  Note that this calculation is analogous with
the one used to derive \eqref{eq:PR_removed2}.

For the converse, we assume that an $F \in \cF_0$ satisfies and
\eqref{eq:annihilate_form}, and we seek an $F' \in N_0 = D_0^\perp$
such that \eqref{eq:asssumption_ker} holds.  We will look for $F'$ of
the form $F' = P_D G+\Lambda P_R F $, which immediately yields
$P_N F' = 0$ and $P_R F' = \Lambda P_R F$, fulfilling
\eqref{eq:asssumption_ker}.

The condition $P_D G+\Lambda P_R F \in D_0^\perp$ is equivalent to
finding $G\in E$ such that
\begin{equation}
  \label{eq:condG1}
  \left<P_D \Psi, G \right> = -\left<\Psi, \Lambda P_R F \right>
\end{equation}
for all $\Psi \in D_0$.  But equation \eqref{eq:condG1} is solvable for
$G$ since the right hand-side is a well-defined functional of
$P_D \Psi$.  Indeed, let $\Psi_1, \Psi_2 \in D_0$ be such that $\Psi_1-\Psi_2
\in \Ker P_D$.  Then $\Psi_1-\Psi_2 \in \cF_0$ and
\eqref{eq:annihilate_form} gives
\begin{equation}
  \label{eq:condG2}
  -\left<\Psi_1-\Psi_2, \Lambda P_R F \right> = 0,
\end{equation}
finishing the proof.

\appendix
\section{Some basic examples}
\label{sec:examples}

In this section we list some basic examples illustrating different
asymptotic behaviors of the eigenvalues on a shrinking graph.  In
fact, in all examples the graph is just an interval.

\begin{example}
  \label{ex:Dirichlet}
  Consider the interval $(0,\epsilon)$ with Dirichlet matching
  conditions,
  \begin{equation}
    \label{eq:VC_Dirichlet}
    f(0) = 0, \qquad f(\epsilon) = 0.
  \end{equation}
  The eigenvalues are $\lambda_n = (\pi n)^2 \epsilon^{-2}$,
  $n\in\mathbb{N}$, i.e.\ all
  eigenvalues are of ``fast'' type,
  Theorem~\ref{thm:main}~\eqref{item:fast}.  To put conditions
  \eqref{eq:VC_Dirichlet} in the form
  \eqref{eq:sa_cond_D}-\eqref{eq:sa_cond_R}, we would set $P_D = I_2$
  (the $2\times2$ identity matrix), $P_N=P_R=0$.
\end{example}

\begin{example}
  \label{ex:Neumann}
  The eigenvalues the interval $(0,\epsilon)$ with Neumann matching
  conditions,
  \begin{equation}
    \label{eq:VC_Neumann}
    f'(0) = 0, \qquad -f'(\epsilon) = 0,
  \end{equation}
  are $\lambda_1 = 0$ and $\lambda_{n+1} = (\pi n)^2 \epsilon^{-2}$,
  $n \in \mathbb{N}$, i.e. all but one are ``fast'' and one is
  ``exotic''.  Conditions~\eqref{eq:VC_Neumann} are equivalent to $P_N
  = I_2$, $P_D=P_R=0$.
\end{example}

\begin{example}
  \label{ex:VC_Robin}
  If we change one of the vertices to Robin condition, namely
  \begin{equation}
    \label{eq:VC_RobinN}
    f'(0) = \gamma f(0), \qquad -f'(\epsilon) = 0,
  \end{equation}
  where $\gamma$ is real, then all but one eigenvalues are ``fast''
  while $\lambda_1 \approx \gamma / \epsilon$, i.e.\ a ``slow''
  eigenvalue with a sign which depends on $\gamma$.

  Notably, when the other vertex is Dirichlet, namely
  \begin{equation}
    \label{eq:VC_RobinD}
    f'(0) = \gamma f(0), \qquad f(\epsilon) = 0,
  \end{equation}
  then for every fixed $\gamma$ all eigenvalues are eventually
  positive and ``fast''.  The influence of the condition at $\epsilon$
  comes through $\Ker P_D$ which, in the case of
  conditions~\eqref{eq:VC_RobinD} is disjoint from $D_0$ resulting in
  $m_0 = \dim\cF_0 = 0$.

  Conditions \eqref{eq:VC_RobinN} are equivalent to
  \begin{equation}
    \label{eq:RD_proj}
    P_D = 0, \quad
    P_N =
    \begin{pmatrix}
      0 & 0 \\ 0 & 1
    \end{pmatrix},
    \quad
    P_R =
    \begin{pmatrix}
      1 & 0 \\ 0 & 0
    \end{pmatrix},
    \quad
    \Lambda =
    \begin{pmatrix}
      \gamma & \cdot \\
      \cdot & \cdot
    \end{pmatrix},
  \end{equation}
  while conditions~\eqref{eq:VC_RobinD} have the above $P_D$ and $P_N$
  swapped.
\end{example}

\begin{example}
  \label{ex:RobinRobin}
  Finally, Robin conditions at both ends may produce more variety,
  including a non-zero exotic eigenvalue $-\gamma^2$ when
  \begin{equation}
    \label{eq:RobinRobin}
    f'(0) = \gamma f(0), \qquad -f'(\epsilon) = -\gamma f(\epsilon).
  \end{equation}
  The latter can be seen to correspond to the eigenfunction $f(x) =
  \exp(\gamma x)$.

  The projectors in this example are
  \begin{equation}
    \label{eq:RR_proj}
    P_D = 0, \quad
    P_N = 0, \quad
    P_R = I_2, \quad
    \Lambda =
    \begin{pmatrix}
      \gamma & 0 \\
      0 & -\gamma
    \end{pmatrix}.
  \end{equation}
\end{example}

\bibliographystyle{myalpha}
\bibliography{bk_bibl,additional}

\end{document}